\documentclass[11pt]{amsart}
\usepackage{amsmath, amsfonts, amssymb}
\usepackage{amscd, amsthm}
\usepackage{geometry}
\usepackage{enumerate}
\usepackage{color}
\usepackage{makecell}
\usepackage{longtable}
\newtheorem{theorem}{Theorem}

\newtheorem{lemma}[theorem]{Lemma}
\newtheorem{corollary}[theorem]{Corollary}
\theoremstyle{remark}
\newtheorem{remark}[theorem]{Remark}

\theoremstyle{definition}

\newcommand{\leng}{\ell}
\newcommand{\horbun}{H}

\newcommand{\Perm}{\operatorname{Perm}}
\newcommand{\dist}{\operatorname{\varrho}}

\newcommand{\length}{\operatorname{length}}
\newcommand{\R}{\mathbb{R}}
\newcommand{\C}{\mathbb{C}}

\newcommand{\lnorm}{\left\Vert}
\newcommand{\rnorm}{\right\Vert}

\newcommand{\lset}{\left\{}
\newcommand{\rset}{\right\}}

\newcommand{\lpar}{\left(}
\newcommand{\rpar}{\right)}

\newcommand{\lip}{\left<}
\newcommand{\rip}{\right>}
\newcommand{\biglpar}{\bigl(}
\newcommand{\bigrpar}{\bigr)}

\newcommand{\Lie}[1]{{\mathfrak{#1}}}
\newcommand{\Centre}{\mathfrak{Z}}

\newcommand{\Span}{\operatorname{span}}  
\newcommand{\id}{\operatorname{Id}}

\newcommand{\ad}{\operatorname{ad}}
\newcommand{\Aut}{\operatorname{Aut}}

\numberwithin{theorem}{section}
\numberwithin{equation}{section}

\begin{document}
\title[Conformal and CR mappings on Carnot groups]{Conformal and CR mappings on Carnot groups}

\author{Michael G.\ Cowling}
\address{Michael G.\ Cowling\\ School of Mathematics and Statistics\\ University of New South Wales\\ UNSW Sydney 2052\\ Australia}
\email{m.cowling@unsw.edu.au}

\author{Ji Li}
\address{Ji Li\\ Department of Mathematics\\ Macquarie University NSW 2109\\ Australia}
\email{ji.li@mq.edu.au}

\author{Alessandro Ottazzi}
\address{Alessandro Ottazzi\\ School of Mathematics and Statistics\\ University of New South Wales\\UNSW Sydney 2052\\ Australia}
\email{a.ottazzi@unsw.edu.au}

\author{Qingyan Wu}
\address{Qingyan Wu\\ Department of Mathematics\\ Linyi University\\ Shandong, 276005, China}
\email{wuqingyan@lyu.edu.cn}
\begin{abstract}
We consider a class of stratified groups with a CR structure and a compatible control distance.
For these Lie groups we show that the space of conformal maps coincide with the space of CR and anti-CR diffeomorphisms.
Furthermore, we prove that on products of such groups, all CR and anti-CR maps are product maps, up to a permutation isomorphism, and affine in each component.
\end{abstract}
\maketitle

\section{Introduction}
In this article,
we consider the interplay between metric  and complex geometry on some model manifolds.
This  is the first outcome of a larger project which aims to develop a unified theory of conformal and CR structures on the one hand, and to define explicit embeddings of certain CR manifolds into $\C^{n}$ on the other.
The analogy between CR and conformal geometry is well documented in the case of CR manifolds of hypersurface type, see, e.g., \cite{Gr, JL, JL2, KR85, Le, Le2, WW2}. The easiest  and perhaps the most studied example in this setting is that of the Heisenberg group, taken with its sub-Riemannian structure.
Here we will focus on those CR manifolds that are stratified groups and that admit a control metric compatible with the CR structure in a suitable sense.

The class of stratified groups that we consider  have a particular algebraic structure, which we call \emph{tight}.
These are the indecomposable examples that mimic the Heisenberg group, in the sense that the metric and algebraic structures are very closely tied together.
It turns out that the only tight stratified groups are either Heisenberg groups or groups whose Lie algebras have two generators (Corollary~\ref{cor:tight}). Tight groups may be endowed with an abstract CR structure. We will show that  the space of CR and anti-CR automorphisms coincides with  the space of conformal maps with respect to a compatible control metric (Theorem~\ref{thm:1qc=CR}).
Hence, we consider the problem of realising our spaces as embedded manifolds.
The fact that these groups embed as CR submanifolds of $\C^{n}$ for appropriate $n$ is a consequence of \cite{AH}, see also \cite{NH}.
We will find explicit embeddings in the cases of free Lie groups with two generators and step at most $8$ (Theorem~\ref{up_to_step8}).
Further, we will show that, on products of tight groups, all CR maps are product maps, up to a permutation isomorphism, and are affine in each component, that is, the composition of a translation with a group automorphism (Theorem~\ref{thm:productCR}).
In order to achieve this, we first show that differentiable quasiconformal mappings on product stratified groups are product mappings, up to an automorphic permutation (Theorem~\ref{main1}).
The last result is a minor variation of \cite[Theorem 1.1]{Xi}.

Here is what follows.
In Section~\ref{prel}, we establish the definitions and properties of Carnot groups and conformal mappings that will be used throughout this paper.
In Section~\ref{CR-str}, we introduce CR structures on stratified groups and define a compatible metric when the groups are tight.
In this section we prove one of our main results, Theorem~\ref{thm:1qc=CR}, and we ask whether we can see these CR groups as boundaries of domains in some $\C^n$. This is equivalent to constructing explicit embeddings in some $\C^n$, which is in turn equivalent to solving a system of differential equations.
In Section~\ref{free-case} we  find the explicit embeddings for the case of free nilpotent groups with two generators up to step $8$, by solving the differential equations of the previous section.
Finally, in Section~\ref{productCR}, we prove our  result about product groups, Theorem~\ref{thm:productCR}, which is a consequence of Theorem~\ref{thm:1qc=CR} and Corollary~\ref{cor:confprod}.

\section{Preliminaries}\label{prel}
In this section, we define stratified Lie algebras and Lie groups, and show how to put sub-Riemannian structures on these.
We also define the Pansu derivative and consider quasiconformal mappings for these structures.

\subsection{Stratified Lie algebras}\label{algebras}

A Lie algebra $\Lie{g}$ is said to be \emph{stratified of step $\leng$} if
\[
\Lie{g}= \Lie{g}_{-1}\oplus \cdots\oplus \Lie{g}_{- \leng},
\]
where $[\Lie{g}_{-j}, \Lie{g}_{-1}] =\Lie{g}_{-j-1}$ when $1\leq j \leq \leng$, while  $\Lie{g}_{-\leng}\neq \{0\}$ and $\Lie{g}_{-\leng-1}=\{0\}$; this implies that $\Lie{g}$ is nilpotent.
We assume that $\dim(\Lie{g})$ is at least $3$ to avoid degenerate cases.

We write $\pi_j$ for the canonical projection of $\Lie{g}$ onto $\Lie{g}_{-j}$, $\Centre(\Lie{g})$ for the centre of $\Lie{g}$, and $\Aut(\Lie{g})$ for the group of automorphisms of $\Lie{g}$.
In particular, for each $s \in \R^+$, the dilation $\delta_s \in \Aut( \Lie{g} )$ is defined to be $\sum_{j=1}^{\leng} s^j \pi_j$.

For a linear map of $\Lie{g}$, preserving all the subspaces $\Lie{g}_{-j}$ of the stratification is equivalent to commuting with dilations, and to having a block-diagonal matrix representation.
We call such maps \emph{strata-preserving}.
We write $\Aut^\delta(\Lie{g})$ for the subset of $\Aut(\Lie{g})$ of strata-preserving automorphisms; these are determined by their action on $\Lie{g}_{-1}$.
A stratified Lie algebra $\Lie{g}$ is said to be \emph{totally nonabelian} if $\Lie{g}_{-1} \cap \Centre(\Lie{g}) = \{0\}$.
If $\Lie{g}$ is totally nonabelian, then $\Lie{g}$ has a \emph{finest direct sum decomposition} (see~\cite[Theorem 2.3]{Cowling-Ottazzi-Product}):
\[
\Lie{g} = \bigoplus_{k=1}^K \Lie{g}^k ,
\]
where the $\Lie{g}^k$ are nontrivial totally nonabelian stratified Lie algebras that commute pairwise, with the additional property that, given any direct sum decomposition $\bigoplus_{l=1}^L \tilde{\Lie{g}}^l$  of $\Lie{g}$ into ideals,  the set $\{ 1, \dots, K \}$ may be partitioned into disjoint subsets $J_1, \dots, J_L$ such that
\[
\tilde{\Lie{g}}^l=\bigoplus_{j\in J_l} \Lie{g}^j
\qquad\forall l\in \{1, \dots, L\}.
\]

When $j, k \in \{1, \dots, K\}$, we write $j \sim k$ if and only if there is a strata-preserving isomorphism from $\Lie{g}^j$ to $\Lie{g}^k$; then $\sim$  is an equivalence relation.
For each equivalence class $[j]$ and each $k \in [j]$, choose a stratified Lie algebra  $\Lie{g}^{[j]}$ isomorphic to $\Lie{g}^j$, and  a strata-preserving isomorphism $I^k$ from $\Lie{g}^{[i]}$ to $\Lie{g}^k$, whose inverse we write as $I^{-k}$.

When $\sigma$ lies in $S_m^{\sim}$, the group of permutations of $\{1, \dots, m \}$ that preserve the equivalence classes of $\sim$, define $I^\sigma \in \Aut^\delta(\Lie{g})$ by first setting
\[
I^{\sigma} (X) = I^{ \sigma(j)} I^{-j}  (X)
\]
for all $X \in \Lie{g}^j$ and all $j \in \{1, \dots, m\}$, and then extending this definition to $\Lie{g}$ by linearity.
It is well known and easy to check that the map $\sigma \mapsto I^{\sigma}$ embeds $S_m^{\sim}$ in $\Aut^\delta(\Lie{g})$.
We denote the image by $\Perm(\Lie{g})$.

\subsection{Stratified Lie  groups}\label{prelim}

Let $G$ be a stratified Lie group of step $\leng$.
This means that $G$ is connected and simply connected, and its Lie algebra $\Lie{g}$ is stratified with $\leng$ layers.
The identity of $G$ is written $e$, and we view the Lie algebra $\Lie{g}$ as the set of left-invariant vectors fields on $G$.

Since $G$ is nilpotent, connected and simply connected, the exponential map $\exp$ is a bijection from $\Lie{g}$ to $G$, with inverse $\log$.
We also write $\delta_s$ for the automorphism of $G$ given by $\exp {} \circ {}{\delta_s} \circ {} \log$.
The differential $T \mapsto (T_*)_e$ is a one-to-one correspondence between automorphisms of $G$  and of $\Lie{g}$, and $T = \exp{} \circ  (T_*)_e \circ {}\log$.
We denote by  $\Aut(G)$ the group of automorphisms of $G$, and by $\Aut^\delta(G)$ the subgroup of automorphisms that commute with dilations.

A stratified connected simply connected Lie group $G$ is called \emph{totally nonabelian} or a \emph{direct product} if its Lie algebra is totally nonabelian or a Lie algebra direct sum.
The \emph{finest direct product decomposition} of the group is that associated to the finest direct sum decomposition of the Lie algebra.

First, we state and prove a preliminary lemma.

\begin{lemma}\label{lem:exponential-coords}
Suppose that $G$ is a simply connected nilpotent Lie group, with an orthonormal basis $\{U_j : j = 1, \dots, n\}$ for its Lie algebra.
Let $\sum_{n=0}^\infty c_n z^n$ be the power series of the function $z / (1 - e^{-z})$ (extended to $0$ by continuity), which converges in the ball with centre $0$ and radius $2\pi$.
Then the left-invariant vector field $X$ corresponding to $X \in \Lie{g}$, evaluated at $\exp(Y)$ in $G$, is given in exponential coordinates of the first kind by
\[
X_{\exp Y}
= \lip \lpar \sum_{k = 0}^\infty c_k \ad^k(Y) \rpar X , U_j \rip \partial_{u_j} ,
\]
where the sum terminates when $k$ is sufficiently large  as $\ad(Y)$ is nilpotent.
\end{lemma}

\begin{proof}
The derivative of the exponential map $\exp$ at $Y \in \Lie{g}$ is given by
\[
\frac{1 - \exp(-\ad(Y))}{\ad(Y)} = \sum_{k=0}^\infty \frac{ (-1)^k }{ (k+1)! } \ad^k(Y)
\]
(see, for example, \cite[Theorem 2.14.3]{Varadajan}), and the series terminates because $G$ is nilpotent.
\end{proof}

The coefficients $c_j$ may be determined inductively from the condition
\[
\lpar \sum_{k=0}^\infty \frac{(-1)^k}{(k+1)!} z^k \rpar \lpar \sum_{k=0}^\infty c_k z^k \rpar = 1,
\]
and we find that
\[
c_0 = 1, \qquad c_1 = \frac{1}{2}\, , \qquad c_2 = \frac{1}{12} \, , \qquad c_3 = 0, \qquad c_4 = \frac{-1}{720} \, , \qquad\dots.
\]

In a stratified group $G$, more can be said.
We say that a function $f$ on $G$ is homogeneous of degree $d$ if $f (\delta_s x) = s^d f(x)$, and a differential operator $D$ on $G$ is homogeneous of degree $e$ (which may be negative) if $Df$ is homogeneous of degree $d+e$ whenever $f$ is homogeneous of degree $d$.
We write $\deg(f)$ and $\deg(D)$ for these degrees.

If $X \in \Lie{g}_{-k}$, then the associated vector field $X$ is homogeneous of degree $-k$.
If we take a basis $\{ U_j : j = 1, \dots, n \}$ of $\Lie{g}$, where each $U_j$ belongs to some $\Lie{g}_{-d(j)}$, and use exponential coordinates of the first kind on $G$, that is, we write
\[
(u_1, \dots, u_n) := \exp(u_1 U_1 + \dots u_n U_n) ,
\]
then the coordinate function $u_j$ is homogeneous of degree $d(j)$.
A vector field that is homogeneous of degree $d$ is a linear combination of the left-invariant vector fields $ U_1$, \dots, $U_n$, with coefficients $c_j$ that are homogeneous functions, and $d = \deg(c_j) + \deg (U_j)$.
This places limits on the polynomials that appear in the various formulae that we use on homogeneous groups.

%

\subsection{The Pansu differential}
We denote by $L_p$ the left translation by $p$ in $G$, that is, $L_p q = pq$ for all $q\in G$.
The subbundle  $\horbun G$ of the tangent bundle $TG$ given by $\horbun_p G = (L_p)_*(\Lie{g}_{-1})$ is called the \emph{horizontal distribution}.
We write $\Omega$ for an arbitrary nonempty connected open subset of $G$.
The differential of a differentiable map $f : \Omega \to G$ is written $f_*$.
We recall that a continuous map $f:\Omega \to G$ is \emph{Pansu differentiable} at $p \in \Omega$ if the limit
\[
\lim_{s\to 0+} \delta_s^{-1}\circ L_{f(p)}^{-1}\circ f \circ L_{p}\circ \delta_s(q)
\]
exists, uniformly for $q$ in compact subsets of $G$;  if it exists, then it is a strata-preserving homomorphism of $G$, written $Df_p(q)$.
If $f$ is Pansu differentiable at $p$, then $\log{}\circ Df_p \circ {}\exp$ is a Lie algebra homomorphism, written  $df_p$, and
\[
df_p(X)=\lim_{s\to 0+} \log {}\circ \delta_s^{-1}\circ L_{f(p)}^{-1}\circ f \circ L_{p}\circ \delta_t \circ \exp (X)
\]
exists, uniformly for $X$ in compact subsets of $\Lie{g}$.
We call  $Df_p$ the \emph{Pansu derivative} and $df_p$ the \emph{Pansu differential} of $f$ at $p$.
By construction, both $Df_p$ and $df_p$ commute with dilations, and so in particular, $df_p$ is a strata-preserving Lie algebra homomorphism.

Note that if $T$ is a strata-preserving automorphism of $G$, then its Pansu derivative ${D}T(p)$ coincides with $T$ at every point, and its Pansu differential $dT(p)$ coincides with the Lie differential $\textrm{log} \circ T \circ \textrm{exp}$ at every point.
Thus our notation is a little different from the standard Lie theory notation, but is not  ambiguous.

\subsection{The sub-Riemannian  distance}
We fix a scalar product $\lip\cdot, \cdot\rip$ on $\Lie{g}_{-1}$, and we define a left-invariant Riemannian metric on the horizontal distribution by the formula
\begin{align}
\lip V, W \rip_p = \lip (L_{p^{-1}})_*(V), (L_{p^{-1}})_*(W) \rip  \label{scalarprod}
\end{align}
for all $V, W \in H_p G$ and all $p \in G$.
This gives rise to a left-invariant \emph{sub-Riemannian} or  \emph{Carnot--Carath\'eodory} distance function $\dist$ on $G$.
To define this, we first say that a smooth curve $\gamma$ is  \emph{horizontal} if $\dot\gamma(t)\in \horbun_{\gamma(t)}G$ for every $t$.
Then we define the distance $\dist(p, q)$ between points $p$ and $q$ by
\[
\dist(p, q) := \inf\int_0^1 \biglpar \lip  \dot\gamma(t), \dot\gamma(t) \rip_{\gamma(t)} \bigrpar^{1/2}  \, dt ,
\]
where in the infimum we take all horizontal curves $\gamma: [0, 1] \to G$ such that $\gamma(0) = p$ and $\gamma(1) = q$.
The distance function is homogeneous, symmetric and left-invariant, that is,
\[
s^{-1} \dist(\delta_s p, \delta_s q)  = \dist(p, q) = \dist(q, p) = \dist(rq, rp)
\qquad\forall p, q, r \in G \quad\forall s \in \R^+;
\]
in particular, $\dist(p, q) = \dist(q^{-1}p, e)$.
The stratified group $G$, equipped with the distance $\dist$, is known as a \emph{Carnot group}; we usually omit mention of $\dist$.

\subsection{Quasiconformal automorphisms and maps}
We write $S(V)$ for the unit sphere in a normed vector space $V$.

Suppose that $\lambda \geq 1$.
We say that $T \in \Aut^\delta(\Lie{g})$ is \emph{$\lambda$-quasiconformal} if and only if
\[
\max \lset \lnorm TX \rnorm : X \in S(\Lie{g}_{-1}) \rset
\leq \lambda \min \lset \lnorm TX \rnorm : X \in S(\Lie{g}_{-1}) \rset .
\]
Of course, every $T \in \Aut^\delta(\Lie{g})$ is $\lambda$-quasiconformal for sufficiently large $\lambda$.

Suppose that $s \in \R^+$.
In a Carnot group, the distortion $H(f, p, s)$ of a map $f: \Omega \to G$ at a point $p \in \Omega$ and at scale $s \in \R^+$ is defined by
\[
H(f, p, s) = \frac{\sup\lset \dist( f(x), f(p) ) : x \in \Omega, \dist(x, p) \leq s \rset }{ \inf\lset \dist( f(x), f(p) ) : x \in \Omega, \dist(x, p) \geq s \rset } \, .
\]
The map $f$ is \emph{$\lambda$-quasiconformal in $\Omega$} if
\[
\limsup_{s \to 0+}  H(f, p, s) \leq \lambda
\qquad\forall p \in \Omega,
\]
and $f$ is \emph{quasiconformal} if it is $\lambda$-quasiconformal for some $\lambda \in \R^+$.

If the map $f$  is $C^1$, then it is $\lambda$-quasiconformal in $\Omega$ if and only if its Pansu differential $df_p$ is $\lambda$-quasiconformal at all $p \in \Omega$.
It is known that $1$-quasiconformal maps on Carnot groups and on some sub-Riemanian manifolds are smooth (see \cite{CCLO, Capogna-Cowling}); such maps are also known as \emph{conformal maps}.

\section{CR stratified groups and Carnot groups}\label{CR-str}
In this section, we consider CR structures on stratified groups and Carnot groups; we consider an example with an illustrious history, and construct many new examples of Carnot groups as boundaries of domains.

\subsection{CR stratified groups}
Let $G$ be a stratified group such that $\dim \Lie{g}_{-1}=2m$  and let $n$ be the integer such that $2m + n = \dim G$.
We define an \emph{almost complex structure} on $\Lie{g}_{-1}$ to be a linear isomorphism $J: \Lie{g}_{-1}\to \Lie{g}_{-1}$ such that
\begin{equation}\label{eqn:AC}
[X, Y] = [JX, JY] \quad\text{and}\quad [X, JY] = -[JX, Y]
\end{equation}
for all $X, Y\in \Lie{g}_{-1}$, whence $J^2=- \id$.
A stratified group $G$ equipped with such a mapping $J$ is said to be a \emph{CR stratified group of type $(m, n)$}.
Let $L=\Span\{X-iJX : X\in \Lie{g}_{-1}\}$.
It is easy to check that \eqref{eqn:AC} is equivalent to $L$ being abelian in the complexification $\Lie{g}_\C$ of $\Lie{g}$.

We say that $T\in \Aut^\delta(\Lie{g})$ is a \emph{CR automorphism} or an \emph{anti-CR automorphism} if
\[
T|_{\Lie{g}_{-1}} J=JT|_{\Lie{g}_{-1}}
\quad\text{or}\quad
T|_{\Lie{g}_{-1}} J=-JT|_{\Lie{g}_{-1}}   ;
\]
equivalently, $T_{\C}(L)\subseteq L$ or $\overline{T_{\C}}(L)\subseteq {L}$, where $T_{\C}$ denotes the linear  extension of $T$ to $\Lie{g}_\C$.
Notice that the inverse of a CR automorphism is also a CR automorphism.
A diffeomorphism $f:\Omega\to \Upsilon$ between domains in $G$ is a \emph{CR mapping} or an \emph{anti-CR mapping} if and only if its Pansu differential $df_p$ is a CR automorphism or an anti-CR automorphism for every $p\in \Omega_1$.
In this section we will study the structure and the CR diffeomorphisms of CR stratified groups.
In particular, we will show that, for a class of these groups, the space of conformal maps with respect to a compatible metric coincides with the space of CR maps.
Last but not least, we will show some explicit embeddings of CR stratified groups into $\C^{m+n}$ via a CR diffeomorphism.

\subsection{Tight groups}
We say that a stratified group is \emph{tight} if $\Lie{g}$ is totally nonabelian, its finest direct product decomposition has only one factor, and $\dim\Lie{g}_{-2} = 1$.
Equivalently,
\[
\Lie{g}_{-1} = \Span\{X_1, \dots, X_m, Y_1, \dots, Y_m\},
\]
where
\begin{equation}\label{def:CR-Tight}
[X_j, X_l] = [Y_j, Y_l]=0
\quad\text{and}\quad
[X_j, Y_l]=\delta_{jl} U
\qquad\forall j, l = 1, \dots, m.
\end{equation}
It is straightforward to check that the space $L$, defined by
\[
L:= \Span\{X_j-iY_j : j = 1, \dots, m\},
\]
satisfies $[L, L]=\{0\}$.
So tight stratified groups are CR with respect to the almost complex structure determined by the requirements that $JX_j = Y_j$ and $JY_j = -X_j$ for every $j=1, \dots, m$.

\begin{lemma}
Let $G$ be a tight CR stratified group.
If $m>1$, then $U$ is central, that is, $[X_j, U]=[Y_j, U]=0$ when $j=1, \dots, m$.
\end{lemma}

\begin{proof}
We argue by contradiction.
Suppose that $[X_j, U]\neq 0$ for some $j$.
Renumbering if necessary, we may assume that $j=1$.
Then $[X_1, [X_2, Y_2]] \neq 0$.
However, by the Jacobi identity and \eqref{def:CR-Tight},
\[
[X_1, [X_2, Y_2]]=[[X_1, X_2], Y_2]+[X_2, [X_1, Y_2]]=0,
\]
which gives a contradiction.
We may show that $[Y_j, U]=0$ similarly.
\end{proof}

\begin{corollary}\label{cor:tight}
Let $G$ be a tight CR stratified group with $\dim (\Lie{g}_{-1})= 2m$.
Then exactly one of the following  holds:
\begin{enumerate}
\item[(i)]
$\Lie{g}_{-1}=\Span\{X_1, Y_1\}$ and $\Lie{g}_{-3}\neq \{0\}$,
\item[(ii)]
$\Lie{g}_{-1}=\Span\{X_1, \dots, X_m, Y_1, \dots, Y_m\}$, $\Lie{g}_{-2}=\Span\{U\}$ and $\Lie{g}_{-3}=\{0\}$.
In this case $\Lie{g}$ is the Heisenberg algebra of dimension $2m+1$.
\end{enumerate}
\end{corollary}

When $G$ is tight, we consider the element $U^*\in \Lie{g}^*$ dual to $U$ and the left-invariant one-form $\theta$ such that $\theta_e=U^*$.
Then the bilinear form
\[
B_\theta(X, Y)=d\theta(X, JY)
\]
is a scalar product on $\Lie{g}_{-1}$ for which
$\{ X_1, \dots, X_m, Y_1, \dots, Y_m\}$
is an orthonormal basis.
Moreover,
$B_\theta$ is \emph{compatible} with $J$, in the sense that
\[
B_\theta(JX, JY)=B_\theta(X, Y)
\qquad\forall X, Y \in \Lie{g}_{-1}.
\]
We  define a Carnot group structure on $G$ using the left-invariant metric on the horizontal subbundle that coincides with $B_\theta$ at the identity.

\begin{theorem}\label{thm:1qc=CR}
Let $G$ be a tight stratified group with the Carnot distance determined by $B_\theta$.
Let $f:\Omega \to G$ be a homeomorphism from a connected open subset $\Omega$ of $G$.
Then $f$ is $1$-quasiconformal if and only if $f$ is CR or anti-CR.
\end{theorem}
\begin{proof}
We say that $T\in \Aut^\delta(\Lie{g})$ is conformal if
\begin{equation}\label{def:confaut}
\| T(X)\|=\lambda \|X\|
\qquad\forall X\in \Lie{g}_{-1}  ,
\end{equation}
or, equivalently, if $T^{t}T=\lambda^2 \id$ for some $\lambda\in \R_+$.
Here $T^{t}$ denotes the transpose with respect to $B_\theta$.
It is well known that $f$ is $1$-quasiconformal if and only if $df_p$ is conformal for every $p\in \Omega$~\cite{Capogna-Cowling}.
Therefore, it is enough to show that every conformal automorphism is a CR or anti-CR automorphism, and vice versa.
Since $G$ is tight, either $\dim \Lie{g}_{-1}=2$ or $G$ is the Heisenberg group of dimension $2m +1$.
It is straightforward to show  in both cases that for all $T\in\Aut^\delta(\Lie{g})$, the condition $T^t|_{\Lie{g}_{-1}}\, T|_{\Lie{g}_{-1}} = \lambda \id$ is equivalent to $T|_{\Lie{g}_{-1}}J = JT|_{\Lie{g}_{-1}}$ or $T|_{\Lie{g}_{-1}}J=-JT|_{\Lie{g}_{-1}}$.
\end{proof}

The theorem above holds for every left-invariant metric that is compatible with the CR structure.
Indeed, let $B'$ be any scalar product on $\Lie{g}_{-1}$ with the property that
\[
B'(JX, JY)=B'(X, Y)
\qquad\forall X, Y \in \Lie{g}_{-1}.
\]
Then there is $A\in GL(2\ell, \R)$ such that $B'(X, Y)=B_\theta (AX, AY)$.
The compatibility condition and the definition of $A$ imply that $B_\theta (JAX, JAY) = B_\theta (AJX, AJY)$, which in turn implies that $AJ = JA$ or $AJ=-JA$.
Therefore $A$ induces a CR or anti-CR automorphism of $\Lie{g}$, and so by Theorem~\ref{thm:1qc=CR} the left-invariant metrics with respect to $B_\theta$ and $B'$ are conformally equivalent.

\subsection{CR embeddings of tight groups}
We are now going to produce explicit CR embeddings of some tight groups $G$ into $\C^{m+n}$ that generalise those considered by Nagel and Stein; our embeddings are more closely related to work of Andreotti and  Hill  \cite{AH}, and generalise work of the fourth-named author and her collaborators \cite{WW1, WW2}.
As the Heisenberg groups are well understood, we concentrate on the case where $\dim \Lie{g}_{-1} = 2$ and $\Lie{g}_{-3}$ is nontrivial.
Our construction involves several steps.

First, extend an orthonormal basis $\{X, Y\}$ of $\Lie{g}_{-1}$ to a basis $\{X, Y, U_1, \dots, U_n\}$ of $\Lie{g}$, where we choose the $U_j$ from the iterated commutators of $X$ and $Y$ in order to first span $\Lie{g}_{-2}$, then $\Lie{g}_{-3}$, and so on.
We then extend the inner product on $\Lie{g}_{-1}$ to an inner product on $\Lie{g}$ so that our basis is orthonormal. We use exponential coordinates of the first kind, and take an element of $G$ to be
\[
(x, y, u_1, \dots u_n) := \exp(xX + yY + u_1U_1 + \dots + u_nU_n).
\]
By Lemma \ref{lem:exponential-coords}, the left-invariant vector field $T$ corresponding to an element $T$ of $\Lie{g}$ is given in these coordinates by
\[
\begin{aligned}
T_{\exp Y}
&= \lip  T , X \rip \partial_{x} + \lip  T , Y \rip \partial_{y}
+ \sum_{j = 1}^n \lip \lpar \sum_{k = 0}^\infty c_k \ad^k(Y) \rpar T , U_j \rip \partial_{u_j} \\
&= a_T \partial_x + b_T \partial_y + \sum_{j=1}^{n} p_{T, j} (x, y, u_1, \dots, u_n) \partial_{u_j},
\end{aligned}
\]
where $a_T = \lip T, X \rip$, $b_t = \lip T, Y\rip$, and $p_{T, j}(x, y, u_1, \dots, u_n)$ is equal to
\[
\begin{split}
\sum_{l = 1}^n \sum_{k=0}^\infty c_k \lip \ad^k(xX + yY + u_1U_1 + \dots + u_n U_n) T, U_l \rip \partial_{u_l}  . \end{split}
\]
The functions $p_{T, j}$ are polynomials of bounded degree, since the series above has finitely many nonzero terms, and are homogeneous if $T$ is homogeneous.
In particular,
\[
X = \partial_x + \sum_{j=1}^{n} p_{X, j}(x, y, u_1, \dots, u_n) \partial_{u_j}
\]
and
\[
Y = \partial_y + \sum_{j=1}^{n} p_{Y, j}(x, y, u_1, \dots, u_n) \partial_{u_j}
\]

Now we seek to map $G$ into the surface
\[
\{ (x, y, u_1, \dots, u_n, v_1, \dots, v_n)) \in \R^{2+2n} :
                v_j = q_j(x, y, u_1, \dots, u_n), j=1, \dots, n \},
\]
where the $q_j$ are homogeneous polynomials of positive degree,
using the map $\phi$, defined by taking $\phi(x, y, u_1, \dots, u_n)$ equal to
\[
 (x, y, u_1, \dots, u_n, q_1(x, y, u_1, \dots, u_n), \dots, q_n(x, y, u_1, \dots, u_n))) .
\]
The map $\phi$ is evidently an embedding, and $0$ lies on the surface.
In the obvious extension of our coordinate system, the differential $\phi_*$ of $\phi$ satisfies
\[
\phi_*(T)
= T + \sum_{j=1}^n (T q_j ) \partial_{v_j}
\]

We identify $(x, y, u_1, \dots, u_n, v_1, \dots, v_n) \in \R^{2+2n}$ with $(z, w_1, \dots, w_n) \in \C^{1+n}$, where $z = x+iy$ and $w_j = u_j + i v_j$.
When we do this, our embedding is a CR embedding if and only if the complex $(1, 0)$  vector field $Z = X + i Y$ on $G$ maps to a $(1, 0)$ vector field tangent to the surface in $\C^{1+n}$.
Now
\[
\phi_*(Z) = \phi_*(X) + i \phi_*(Y),
\]
and this is a $(1, 0)$ vector field if and only if the coefficient of $\partial_{v_j}$ is $i$ times the coefficient of $\partial_{u_j}$ for all $j$, that is,
\[
\begin{aligned}
(X q_j )  + i (Y q_j )  & = i\biglpar p_{X, j} + i p_{Y, j} \bigrpar ,
\end{aligned}
\]
or equivalently,
\begin{equation}\label{embedding}
\begin{cases}
X q_j   &=     - p_{Y, j}(x, y, u_1, \dots, u_n)  \\
Y q_j   &=        p_{X, j}(x, y, u_1, \dots, u_n)  .
\end{cases}
\end{equation}
This is a nontrivial system of differential equations, and we do not know whether it can be solved in general. There are certainly many examples where this is possible, for instance, if $\Lie{g}$ is filiform---see \cite{WW1, WW3}.
In the next section, we will consider the case of free nilpotent Lie groups with two generators and use an alternative coordinate system to solve \eqref{embedding} for the case when these Lie groups have step at most $8$.

If the stratified group $G$ is not tight, there is no obvious canonical choice of a compatible Carnot structure.
Hence the extent to which we can generalise our study of the interplay between conformal and CR structures in the general case is unclear.


\section{Free nilpotent Lie groups}\label{free-case}
In this section we focus on solving the system of equations \eqref{embedding} in the tight case, that is, for free nilpotent Lie groups whose Lie algebra has two generators and step at least $2$.
We introduce some further notation, that in some cases will replace that of the previous sections.
Denote by $\mathfrak{f}_{2,s}$  the free nilpotent Lie algebra of step $s$ with $2$ generators, and let $n= \dim \mathfrak{f}_{2,s}$.
Recall that $\mathfrak{f}_{2,s}$ is the biggest nilpotent Lie algebra of step $s$ generated by iterated brackets of two generators $X_1$ and $X_2$. Given vectors $X_1,\dots,X_\ell$ in $\mathfrak{f}_{2,s}$, the elements in the linear span of
\[
[X_{\alpha_1},\dots,[X_{\alpha_{k-1}},X_{\alpha_k}]\dots],
\]
where $1\leq \alpha_i \leq \ell$, are said to have \emph{length at most $k$}.
We now recall the recursive definition of the Hall basis \cite{Hall} for $\mathfrak{f}_{2,s}$.
Each element in the basis is a monomial in the generators.
The generators $X_1$ and $X_2$ are elements of the basis and of length $1$. Assume that we have defined basis elements of lengths $1,\dots, \ell -1$ and that they are simply ordered in such a way that $X<Y$ if $\length(X)<\length(Y)$.
If $\length(X)=r$ and $\length(Y)=t$, and $\ell = r+t$, then $[X,Y]$ is a basis element of length $\ell$ if:
\begin{enumerate}[1.]
 \item  $X$ and $Y$ are basis elements and $X>Y$, and
 \item  if $X=[Z,W]$, then $Y\geq W$.
\end{enumerate}

Number the basis elements using this ordering, i.e., $X_3=[X_2,X_1]$, $X_4=[X_3,X_1]$, $X_5=[X_3,X_2]$, etc. Consider a basis element $X_i$ as a bracket in the lower order basis elements, $[X_{j_1},X_{k_1}]$, where $j_1>k_1$. If we repeat this process with $X_{j_1}$, we get $X_i = [[X_{j_2},X_{k_2}],X_{k_1}]$, where $k_2\leq k_1$ by the Hall basis conditions. Continuing in this fashion, we end up with
 \begin{equation}\label{decomposed_bracket}
 X_i=[[\dots[[X_2,X_{i_1}],X_{i_2}],\dots,X_{i_{m-1}}],X_{i_m}],
 \end{equation}
where $i_1=1$ and $i_\ell \leq i_{\ell+1}$ when $2\leq \ell\leq m-1$. Since this expansion involves $m$ brackets, we shall write $d(i)=m$ and define $d(1)=d(2)=0$. This process naturally associates a multi-index $I(i)=(a_1,\dots,a_n)$ to each Hall basis element $X_i$, defined by $a_r = \# \{t: i_t=r\}$. Note that $I(i)=(0,\dots,0)$ for $i=1,2$. Let $x=(x_1,\dots,x_n)$ be coordinates in $\R^n$.
For every $j\geq 3$, we define the monomial $p_{j}$
by
\[
p_{j}(x) := \frac{(-1)^{d(j)}}{I(j)!} x^{I(j)},
\]
where $x^{I(j)}= x_1^{I(j)_1}\cdots x_n^{I(j)_n}$ and  $I(j)!=I(j)_1!\cdots I(j)_n!$, if $I(j)=(I(j)_1,\dots,I(j)_n)$.
Notice that $I(j)_1\geq 1$.
It will be convenient to represent the bracket \eqref{decomposed_bracket} by the vector $(2,1,i_2,\dots,i_m)$.
We stress that knowing any one of the formula \eqref{decomposed_bracket} for $X_i$, the vector $(2,1,i_2,\dots,i_m)$, or the monomial $p_{i}$, uniquely describes the other two.

The vector fields
\[
X_1 = \frac{\partial}{\partial{x_1}} \quad \text{and}\quad X_2 =  \frac{\partial}{\partial{x_2}} + \sum_{j\geq 3}p_{j} \frac{\partial}{\partial{x_j}}
\]
generate the Lie algebra $\mathfrak{f}_{2,s}$.
We now rewrite \eqref{embedding} using these vector fields as generators. Thus, we look for polynomials $q_j$ solving
\[
\begin{cases}
X_1 q_j&=-p_{j}\\
X_2 q_j &=0
\end{cases}
\]
for every $j\geq 3$.
Since $p_j$ is a monomial, the first equation integrates to
\[
q_j = c_j x_1 p_{j} + r_j,
\]
where $c_j\in(0,1]$ and $r_j = r_j(x_2,\dots,x_n)$.
We substitute this in the second equation to obtain
\[
c_jx_1X_2p_{j} + X_2 r_j = c_jx_1\left( \frac{\partial}{\partial{x_2}}p_{j} + \sum_{k>2}p_{k}\frac{\partial}{\partial{x_k}}p_{j} \right)+\frac{\partial}{\partial{x_2}}r_j+ \sum_{\ell>2}p_{\ell}\frac{\partial}{\partial{x_\ell}}r_j=0.
\]
Since $I(t)_1\geq 0$ when $t\geq 3$, it follows that $\partial r_j/ \partial{x_2}$ is the only term that does not depend on $x_1$, so it is zero, and $r_j=r_j(x_3,\dots,x_n)$.
Hence for a free nilpotent Lie algebra with two generators, the system  \eqref{embedding} can be solved if we can find $r_j(x_3,\dots,x_n)$ such that
\begin{equation}\label{integrateforfree}
c_jx_1\left( \frac{\partial}{\partial{x_2}}p_{j} + \sum_{k>2}p_{k}\frac{\partial}{\partial{x_k}}p_{j} \right)+ \sum_{\ell>2}p_{\ell}\frac{\partial}{\partial{x_\ell}}r_j=0
\end{equation}
for all $j\geq 3$.
We now solve this system of equations for free Lie algebras up to step $8$.

\begin{theorem}\label{up_to_step8}
Let $\mathfrak{f}_{2,s}$ be a free nilpotent Lie algebra of step $s$ at most $8$.
Then the system of equations \eqref{integrateforfree} admits a solution of the form
\[
r_j(x_3,\dots,x_n) = \sum_{k>2} a_k^j x_k + \sum_{\ell>2} b_\ell^j x_\ell^2,
\]
for some $a_k^j,b_\ell^j\in \R$. In particular, if $s\leq  5$, then $b_\ell^j=0$.

Vice versa, if $s\geq 9$, then there is $j\geq 3$ such that $\sum_{k>2} a_k^j x_k + \sum_{\ell>2} b_\ell^j x_\ell^2$ does not solve \eqref{integrateforfree}.
\end{theorem}
For a free nilpotent Lie algebra with two generators, we may represent the monomials $p_{\ell}$ by the vector
$(2,1,\ell_1,\dots,\ell_m)$.  We stress that this vector is not the same as $I(\ell)$.
Vice versa, in order for such a vector to represent a nonzero monomial $p_{\ell}$, it must be that
$(2,1,\ell_1,\dots,\ell_{m-1})$ represents a basis vector higher than $X_{\ell_m}$ and with $X_{\ell_m}> X_{\ell_{m-1}}$.
Using these rules, we may easily construct all vectors for a given step. For example, if the step is two, we only have the monomial $p_{3}$, corresponding to the vector $(2,1)$. In step three, we have to add
 two more monomials, $p_{4}$, and $p_{5}$, corresponding to the vectors $(2,1,1)$ and $(2,1,2)$.
We  write all  vectors and monomials up to step eight in Table~\ref{uptostep8table}.
\begin{table}[t]
\scriptsize
\caption{Monomials up to step $8$.}
\begin{center}
\begin{tabular}{|c|l|c|}
\hline & &\\
Step & Vectors & Monomials\\
\hline   & & \\
$2$ & $(2,1)$ &  \makecell[l]{$p_{3} = -x_1$}\\
\hline & &\\
$3$ & $(2,1,1), \,(2,1,2)$ &  \makecell[l]{$p_{4} = \frac{1}{2}x_1^2,\, p_{5}=x_1x_2$}\\
\hline & &\\
$4$ & \makecell[l]{$(2,1,1,1),\,(2,1,1,2)$\\$(2,1,2,2)$} & \makecell[l]{$p_{6}=-\frac{1}{6}x_1^3, p_{7}=-\frac{1}{2}x_1^2x_2$\\$p_{8}=-\frac{1}{2}x_1x_2^2 $}\\
\hline & & \\
$5$ & \makecell[l]{$(2,1,1,1,1),\,(2,1,1,1,2)$\\$(2,1,1,2,2),\,(2,1,2,2,2)$\\$(2,1,1,3),\,(2,1,2,3)$} & \makecell[l]{$p_{9}=\frac{1}{24}x_1^4, p_{10}=\frac{1}{6}x_1^3x_2$\\$p_{11}=\frac{1}{4}x_1^2x_2^2,p_{12}=\frac{1}{6}x_1x_2^3$\\$p_{13}=-\frac{1}{2}x_1^2x_3,p_{14}=-x_1x_2x_3$}\\
\hline & &\\
$6$ & \makecell[l]{$(2,1,1,1,1,1),\,(2,1,1,1,1,2)$\\$(2,1,1,1,2,2),\,(2,1,1,2,2,2)$\\$(2,1,2,2,2,2),\,
(2,1,1,1,3)$\\$(2,1,1,2,3),\,(2,1,2,2,3)$\\$(2,1,2,4)$} & \makecell[l]{$p_{15}=-\frac{1}{120}x_1^5,p_{16}=-\frac{1}{24}x_1^4x_2$\\$p_{17}=-\frac{1}{12}x_1^3x_2^2,p_{18}=-\frac{1}{12}x_1^2x_2^3$\\$p_{19}=-\frac{1}{24}x_1x_2^4,p_{20}=\frac{1}{6}x_1^3x_3$\\$p_{21}=\frac{1}{2}x_1^2x_2x_3,p_{22}=\frac{1}{2}x_1x_2^2x_3$\\$p_{23}=-x_1x_2x_4$}\\
\hline & &\\
$7$ &  \makecell[l]{$(2,1,1,1,1,1,1),\,(2,1,1,1,1,1,2)$\\
$(2,1,1,1,1,2,2),\,(2,1,1,1,2,2,2)$\\
$(2,1,1,2,2,2,2),\,(2,1,2,2,2,2,2)$\\
$(2,1,1,1,1,3),\,(2,1,1,1,2,3)$\\
$(2,1,1,2,2,3),\,(2,1,2,2,2,3)$\\
$(2,1,1,3,3),\,(2,1,2,3,3)$\\
$(2,1,1,1,4),\,(2,1,1,1,5)$\\
$(2,1,1,2,4),\,(2,1,1,2,5)$\\
$(2,1,2,2,4),\,(2,1,2,2,5)$} & \makecell[l]{
$p_{24}=\frac{1}{720}x_1^6,p_{25}=\frac{1}{120}x_1^5x_2$\\
$p_{26}=\frac{1}{48}x_1^4x_2^2,p_{27}=\frac{1}{36}x_1^3x_2^3$\\
$p_{28}=\frac{1}{48}x_1^2x_2^4,p_{29}=\frac{1}{120}x_1x_2^5$\\
$p_{30}=-\frac{1}{24}x_1^4x_3,p_{31}=-\frac{1}{6}x_1^3x_2x_3$\\
$p_{32}=-\frac{1}{4}x_1^2x_2^2x_3, p_{33}=-\frac{1}{6}x_1x_2^3x_3$\\
$p_{34}=\frac{1}{4}x_1^2x_3^2, p_{35}=\frac{1}{2}x_1x_2x_3^2$\\
$p_{36}=\frac{1}{6}x_1^3x_4,p_{37}=\frac{1}{6}x_1^3x_5$\\
$p_{38}=\frac{1}{2}x_1^2x_2x_4, p_{39}=\frac{1}{2}x_1^2x_2x_5$\\
$p_{40}=\frac{1}{2}x_1x_2^2x_4,p_{33}=\frac{1}{2}x_1x_2^2x_5$}\\
\hline & &\\
$8$ &  \makecell[l]{$(2,1,1,1,1,1,1,1),\,(2,1,1,1,1,1,1,2)$\\
$(2,1,1,1,1,1,2,2),\,(2,1,1,1,1,2,2,2)$\\
$(2,1,1,1,2,2,2,2),\,(2,1,1,2,2,2,2,2)$\\
$(2,1,2,2,2,2,2,2),\,(2,1,1,1,1,1,3)$\\
$(2,1,1,1,1,2,3),\,(2,1,1,1,2,2,3)$\\
$(2,1,1,2,2,2,3),\,(2,1,2,2,2,2,3)$\\
$(2,1,1,1,3,3),\,(2,1,1,2,3,3)$\\
$(2,1,2,2,3,3),\,(2,1,1,1,1,4)$\\
$(2,1,1,1,1,5),\,(2,1,1,1,2,4)$\\
$(2,1,1,1,2,5),\,(2,1,1,2,2,4)$\\
$(2,1,1,2,2,5),\,(2,1,2,2,2,4)$\\
$(2,1,2,2,2,5),\,(2,1,1,3,4)$\\
$(2,1,1,3,5),\,(2,1,2,3,4)$\\
$(2,1,2,3,5),\,(2,1,1,2,6)$\\
$(2,1,2,2,6),\,(2,1,2,2,7)$} & \makecell[l]{
$p_{41}=-\frac{1}{5040}x_1^7,p_{42}=-\frac{1}{720}x_1^6x_2$\\
$p_{43}=-\frac{1}{240}x_1^5x_2^2,p_{44}=-\frac{1}{144}x_1^4x_2^3$\\
$p_{45}=-\frac{1}{144}x_1^3x_2^4,p_{46}=-\frac{1}{240}x_1^2x_2^5$\\
$p_{47}=-\frac{1}{720}x_1x_2^6,p_{48}=\frac{1}{120}x_1^5x_3$\\
$p_{49}=\frac{1}{24}x_1^4x_2x_3, p_{50}=\frac{1}{12}x_1^3x_2^2x_3$\\
$p_{51}=\frac{1}{12}x_1^2x_2^3x_3, p_{52}=\frac{1}{24}x_1x_2^4x_3$\\
$p_{53}=-\frac{1}{12}x_1^3x_3^2,p_{54}=-\frac{1}{4}x_1^2x_2x_3^2$\\
$p_{55}=-\frac{1}{4}x_1x_2^2x_3^2, p_{56}=\frac{1}{24}x_1^4x_4$\\
$p_{57}=-\frac{1}{24}x_1^4x_5,p_{58}=-\frac{1}{6}x_1^3x_2x_4$\\
$p_{59}=-\frac{1}{6}x_1^3x_2x_5,p_{60}=-\frac{1}{4}x_1^2x_2^2x_4$\\
$p_{61}=-\frac{1}{4}x_1^2x_2^2x_5,p_{62}=-\frac{1}{6}x_1x_2^3x_4$\\
$p_{63}=-\frac{1}{6}x_1x_2^3x_5,p_{64}=\frac{1}{2}x_1^2x_3x_4$\\
$p_{65}=\frac{1}{2}x_1^2x_3x_5,p_{66}=x_1x_2x_3x_4$\\
$p_{67}=x_1x_2x_3x_5, p_{68}=\frac{1}{2}x_1^2x_2x_6$\\
$p_{69}=\frac{1}{2}x_1x_2^2x_6, p_{70}=\frac{1}{2}x_1x_2^2x_7$\\}\\
\hline
\end{tabular}
\end{center}
\label{uptostep8table}
\end{table}%

Theorem~\ref{up_to_step8} will be a  consequence of the following three lemmas.

\begin{lemma}\label{polynomial_idem}
Let $\mathfrak{f}_{2,s}$ be a free nilpotent Lie algebra of step $s$ at most $8$.
Then for all $j,k\geq 3$,
\[
x_1p_{k} \frac{\partial}{\partial{x_k}}p_{j} \in \Span\{p_{\ell}\},
\]
for some $\ell=\ell(j,k)$.
\end{lemma}
\begin{proof}
We say that $X_j$ has height $h(j)$ if $X_j\in \mathfrak{g}_{-h(j)}$.
Observe that  $p_{j}$ is homogeneous of degree $h(j)-1$.
It follows that,  if $x_1p_{k} \frac{\partial}{\partial{x_k}}p_{j}\neq 0$, then it is homogeneous of the same degree as $p_{j}$.
If this is the case, for our statement to be true, it must be that $h(\ell)=h(j)$.
Moreover, $\frac{\partial}{\partial{x_k}}p_{j}$ can only be nonzero if $3\leq k\leq 7$.
Using these observations, the claim readily follows by inspecting Table~\ref{uptostep8table}.
\end{proof}

\begin{lemma}\label{polynomial_second}
Let $\mathfrak{f}_{2,s}$ be a free nilpotent Lie algebra of step $s$ at most $8$.
Then for every $j\geq 3$,
\[
x_1\frac{\partial}{\partial{x_2}}p_{j} \in
\begin{cases}
\Span\{p_{\ell}\}, \quad \text{if } j\neq 23,\\
 \Span\{x_4 p_{4}\}, \quad \text{if } j= 23\\
 \end{cases}
\]
for some $\ell=\ell(j)$.
In particular, if $s\leq 5$, then $x_1\frac{\partial}{\partial{x_2}}p_{j} \in\Span\{p_{\ell}\}$.
\end{lemma}
\begin{proof}
Given $p_{j}$, we study the action of
 $x_1\frac{\partial}{\partial{x_2}}$ on the associated vector $(2,1,j_2,\dots,j_m)$.
If  $j_k\neq 2$ for all $k$, then  $x_1\frac{\partial}{\partial{x_2}}p_{j}=0$ and we are done.
Otherwise, let $k$ be the smallest integer such that $j_k=2$. Then the action of $x_1\frac{\partial}{\partial{x_2}}$ replaces the
$j_k$th entry with $1$, and more precisely,
\[
(2,1,j_2,\dots,j_m) \mapsto (2,1,j_2,\dots,j_{k-1},1,j_{k+1},\dots,j_m).
\]
If $(2,1,j_2,\dots,j_{k-1},1,j_{k+1},\dots,j_m)$ represents a monomial $p_{\ell}$ for some $\ell$, then we are done.
By inspecting  Table~\ref{uptostep8table}, this always occurs except for the vector $(2,1,2,4)$ in step $6$, which represents the monomial $p_{23}$.
In this final case,
 \[
x_1\frac{\partial}{\partial{x_2}} p_{23}=  -x_1\frac{\partial}{\partial{x_2}}x_1x_2x_4 = -x_1^2 x_4= -x_4 p_{4}. \\
\]

\end{proof}

\begin{lemma}\label{counter-example}
Let $\mathfrak{f}_{2,s}$ be a free nilpotent Lie algebra of step $9$.
Then there is $j$ for which $ \sum_{k>2} a_k^j x_k + \sum_{\ell>2} b_\ell^j x_\ell^2$ is not a solution of \eqref{integrateforfree}, for every $a_k^j,b_\ell^j \in\R$.
\end{lemma}

\begin{proof}
When $s = 9$, consider the monomial $p_{j}=x_1x_2x_4x_5$, corresponding to the vector $(2,1,2,4,5)$.
A direct computation shows that every solution $r_j$ of  \eqref{integrateforfree} for this monomial contains $c_j x_4^2 x_5$, for some $c_j\neq 0$.
\end{proof}

\begin{proof}[Proof of Theorem~\ref{up_to_step8}]
In view of Lemmas~\ref{polynomial_idem} and~\ref{polynomial_second}, if $j\neq 23$ we may rewrite the equation
\eqref{integrateforfree} as
\begin{align*}
\sum_{v>2}a_vp_{v} +\sum_{\ell>2}p_{\ell} \frac{\partial}{\partial x_\ell}r_j =\sum_{\ell>2}\left(a_\ell +\frac{\partial}{\partial x_\ell}r_j\right)p_{\ell}=
0
\end{align*}
for some constants $a_v$. Then $r_j = -\sum_{\ell>3} a_\ell x_\ell$ is a solution.
If $j=23$, then \eqref{integrateforfree} becomes
\[
-c_4x_4p_{4}+  \sum_{v>2,v\neq 4}b_vp_{v} +\sum_{\ell>2}p_{\ell} \frac{\partial}{\partial x_\ell}r_{23}=0
\]
for some constants $b_v$. Hence $r_{23}=-\sum_{v>2,v\neq 4}b_vx_v +\frac{c_4}{2}x_4^2$ is a solution.
\end{proof}

Theorem~\ref{up_to_step8} suggests that all free tight groups can be embedded into $\C^{1+n}$ by means of polynomial functions.
We plan to investigate these embeddings, and embeddings of more general tight groups, seen as quotients of the free ones, in a subsequent paper.


\section{Products}\label{productQC}
We say that a mapping on $G$ is  \emph{affine} if it is the composition of a left translation with an element in $\Aut^\delta(G)$.

\begin{theorem}\label{main1}
Suppose that $G$ is a totally nonabelian Carnot group, with finest direct product decomposition   $G^1\times \dots \times G^m$, where $m>1$.
Let $f:G\to G$ be a $C^1$  quasiconformal map.
Then $f$ is composed of a group automorphism that permutes the groups $G^j$ and a product bi-Lip\-schitz map.
\end{theorem}

\begin{proof}
We write either $p$ or $(p^1, \dots, p^m)$ for a typical element of $G$.

The Pansu differential of a $C^1$ global quasiconformal mapping $f$ is continuous, and hence its $\Perm(\Lie{g})$ component is constant.
This is an automorphism of $\Lie{g}$ and by conjugation with the exponential may be considered as an automorphism of $G$, and is therefore quasiconformal.
By composing with the inverse of this automorphism if necessary, we may assume that the Pansu differential $df_p$ of $f$ is a product automorphism (see \cite[Corollary 3.4]{Cowling-Ottazzi-Product}).

If we take a horizontal curve $\gamma$ in one of the factors $G^j$, then $f \circ \gamma$ is again a horizontal curve, whose Pansu derivative is $df \circ \dot\gamma$, and so $f \circ \gamma$ moves in the factor $G^j$ and is fixed in the other factors.
The groups $G^j$ mutually commute, and it follows immediately that $f$ is a product map: we may find maps $f^j: G^j \to G^j$ such that
\[
f(p^1, \dots, p^m) = (f^1(p^1), \dots, f^m(p^m))
\qquad\forall p \in G.
\]

The Pansu differential $df_p$ is also a product map:
\[
df_{p} = (df^1_{p^1}, \dots, df^m_{p^m})
\qquad\forall p \in G.
\]
If $f$ is $\lambda$-quasiconformal, it follows immediately that when $j \neq k$,
\[
\max\big\{ \big \| df^j_{p^j}(X) \big\| : X \in S(\Lie{g}_{-1}^j) \big\}
\leq \lambda \min\left\{  \lnorm df^k_{p^k}(X) \rnorm : X \in S(\Lie{g}_{-1}^k) \right\}
\]
for all $p \in G$.
Define
\[
c_k =  \inf\left\{  \lnorm df^k_{p^k}(X) \rnorm : X \in S(\Lie{g}_{-1}^k), p^k \in G^k \right\}
\]
for all $k$, and now fix $k$ such that $c_k = \min \{ c_j : j \in \{1, \dots, m\} \}$.
Then, when $j\neq k$,
\[
\sup\big\{  \big\| df^j_{p^j}(X) \big\| : X \in S(\Lie{g}_{-1}^j) , p^j \in G^j  \big\} \leq \lambda c_k.
\]
Fix $j $ different to $k$.
Since
\[\begin{split}
\sup\left\{  \lnorm df^k_{p^k}(X) \rnorm : X \in S(\Lie{g}_{-1}^k) , p^k \in G^k \right\} \qquad\qquad \\
\qquad\qquad \leq \lambda \inf\big\{  \| df^j_{p^j}(X) \| : X \in S(\Lie{g}_{-1}^j) , p^j \in G^j \big \},
\end{split}\]
it follows that
\[
\sup\{  \lnorm df^k_{p^k}(X) \rnorm : X \in S(\Lie{g}_{-1}^k) , p^k \in G^k  \} \leq \lambda^2 c_k.
\]
Since $f$ is not constant, $c_k \neq 0$, and now each map $f^j$ is bi-Lip\-schitz, and, by considering horizontal curves, we conclude that
\[
c_k \dist(p, q) \leq \dist(f(p), f(q)) \leq \lambda^2 c_k \dist(p, q)
\qquad\forall p, q \in G,
\]
as required.
\end{proof}

\begin{remark}
The argument above shows that, if $f$ is defined in a domain $\Omega$ in $G$, then $f$ is locally a product mapping.
Of course, this does not imply that $f$ is a product mapping, unless $\Omega$ is a product domain.
However, if $f$ is $1$-quasiconformal, then stronger conclusions do hold.
\end{remark}

\begin{corollary}\label{cor:confprod}
Suppose that $G$ is a totally nonabelian Carnot group, with finest direct product decomposition   $G^1\times \dots \times G^m$, where $m>1$.
Let $f:\Omega \to G$ be a  $1$-quasiconformal map from a domain $\Omega$ in $G$ onto its image.
Then $f$ is the restriction to $\Omega$ of the composition of a group automorphism that permutes the groups $G^j$ and a product affine map.
\end{corollary}

\begin{proof}
By \cite{Capogna-Cowling}, $f$ is smooth.
By \cite[Theorem 4.1]{Cowling-Ottazzi-Conformal}, $f$ is an affine map.
In particular, $f$ extends analytically to a conformal map on all of $G$.
By Theorem~\ref{main1}, it follows that $f$ is a product map.
\end{proof}

\begin{remark}
We recall that if $G$ is the Heisenberg group $H^n$, then conformal maps on a domain in $G$ are restrictions of the action of an element of $SU(1, n+1)$ \cite{KR85}.
However, if $G$ is the product of $m$ Heisenberg groups $H^{n_l}$ where $m\geq 2$, then most elements in $SU(1, n_1+1)\times\dots\times SU(1, n_m+1)$ do not induce conformal maps on domains in $G$.
Indeed, from the previous corollary, conformal maps are affine in this case.
\end{remark}

\subsection{CR mappings on product groups}

\begin{theorem}\label{thm:productCR}\label{productCR}
Suppose that $G$ is a totally nonabelian Carnot group, with finest direct product decomposition   $G^1\times \dots \times G^m$, where $m>1$.
Suppose that $G^j$ is tight when $j=1, \dots, m$.
Let $f:G\to G$ be a CR mapping.
Then $f$ is the composition of a group automorphism that permutes the groups $G^j$ and a product CR mapping. \end{theorem}

\begin{proof}
By Theorem~\ref{thm:1qc=CR}, for tight Carnot groups CR and anti-CR diffeomorphisms are the same as conformal mappings.
The conclusion now follows from Corollary~\ref{cor:confprod}.
\end{proof}

\end{document}